\documentclass[12pt]{article}
\usepackage{amssymb}
\usepackage{amsmath}
\usepackage{amsthm}
\usepackage{yhmath}
\usepackage{mathdots}
\usepackage{MnSymbol}
\usepackage{a4wide}
\usepackage{color}
\usepackage[left=2.5cm,top=2cm,right=2cm]{geometry}
\usepackage{enumitem}
\usepackage{nccmath}
\usepackage{blkarray}
\usepackage{multirow}
\usepackage{float}
\usepackage{booktabs}
\usepackage{arydshln}
\usepackage[square,numbers]{natbib}
\usepackage{graphicx}
\graphicspath{ {/Users/neha/Desktop/NLA P4/Paper4} }
\usepackage{subfigure}

\usepackage{natbib}
\usepackage[english]{babel}
\usepackage{algorithm}
\usepackage{algpseudocode}
\newtheorem{problem}{\bf Problem}[section]
\newtheorem{theorem}{\bf Theorem}[section]
\newtheorem{definition}[theorem]{\bf Definition}
\newtheorem{corollary}[theorem]{\bf Corollary}
\newtheorem{exam}[theorem]{\bf Example}
\newtheorem{remark}[theorem]{\bf Remark}
\newtheorem{lemma}[theorem]{\bf Lemma}

\SetLabelAlign{CenterWithParen}{\hfil(\makebox[1.0em]{#1})\hfil}
\def \R{{\mathbb R}}
\def \C{{\mathbb C}}
\def \SR{{\mathbb{S}\mathbb{R}}}
\def \ASR{{\mathbb{A}\mathbb{S}\mathbb{R}}}
\def \rank{\mathrm{rank}}
\def \vec{\mathrm{vec}}
\def \QR{\mathbb{Q}_\mathbb{R}}
\def \HQR{\mathbb{H}\mathbb{Q}_\mathbb{R}}
\def \HC{\mathbb{H}\mathbb{C}}
\def \i{\textit{\textbf{i}}}
\def \j{\textit{\textbf{j}}}
\def \k{\textit{\textbf{k}}}
\newcommand\norm[1]{\left\lVert#1\right\rVert}
\newcommand{\Rn}[1]{%
	\textup{\lowercase\expandafter{\romannumeral#1}}%
}

\def \diag{\mathrm{diag}}

\def \noin{\noindent}
\newcommand{\beano}{\begin{eqnarray*}}	\newcommand{\eeano}{\end{eqnarray*}}
\renewcommand{\thefootnote}{\fnsymbol{footnote}}

\date{\today}
\title{Special least squares solutions of the reduced biquaternion matrix equation with applications
	
	\footnotemark[2]}
\author{Sk. Safique Ahmad\footnotemark[1] \and Neha Bhadala \footnotemark[3]}

\begin{document}
	\maketitle
	\begin{abstract}
This paper presents an efficient method for obtaining the least squares Hermitian solutions of the reduced biquaternion matrix equation $(AXB, CXD) = (E, F)$. The method leverages the real representation of reduced biquaternion matrices. Furthermore, we establish the necessary and sufficient conditions for the existence and uniqueness of the Hermitian solution, along with a general expression for it. Notably, this approach differs from the one previously developed by Yuan et al. $(2020)$, which relied on the complex representation of reduced biquaternion matrices. In contrast, our method exclusively employs real matrices and utilizes real arithmetic operations, resulting in enhanced efficiency. We also apply our developed framework to find the Hermitian solutions for the complex matrix equation $(AXB, CXD) = (E, F)$, expanding its utility in addressing inverse problems. Specifically, we investigate its effectiveness in addressing partially described inverse eigenvalue problems. Finally, we provide numerical examples to demonstrate the effectiveness of our method and its superiority over the existing approach. 
	\end{abstract}
	
	
	
	\renewcommand{\thefootnote}{\fnsymbol{footnote}}
	
	\footnotetext[1]{
		Department of Mathematics, Indian Institute of Technology Indore, Simrol, Indore-452020, Madhya Pradesh, India. \texttt{email:safique@iiti.ac.in, safique@gmail.com}.}
	
	\footnotetext[3]{Research Scholar, Department of Mathematics, IIT Indore, Research work funded by PMRF (Prime Minister's Research Fellowship). \texttt{email:phd1901141004@iiti.ac.in, bhadalaneha@gmail.com}}
	\section{Introduction}\label{sec1}
	Matrix equations play a vital role in various fields, including computational mathematics, vibration theory, stability analysis, and control theory (for example, \cite{datta2004numerical,gantmacher2005applications,MR1407657}). Due to their significant applications in various fields, a substantial amount of research has been dedicated to solving matrix equations. See \cite{MR1976921,MR0279115,MR1972682,MR1637924} and references therein. In this paper, we address the following matrix equation: 
	\begin{equation}\label{eq1.1}
	(AXB, CXD) = (E, F), 
	\end{equation}
	where $A$, $B$, $C$, $D$, $E$, and $F$ are given matrices of appropriate sizes, and $X$ represents the unknown matrix of an appropriate size. This matrix equation has been extensively studied in the real, complex, and quaternion fields, leading to several important findings. Notably, Liao et al. \cite{MR2165440} provided an analytical expression for the least squares solution with the minimum norm in the real field. Ding et al. \cite{MR2646321} proposed two iterative algorithms to obtain solutions in the real field. Mitra \cite{MR0320028} delved into the necessary and sufficient conditions for the existence of solution and presented a general formulation for the solution in the complex field. Navarra et al. \cite{MR1826896} offered a simpler consistency condition and a new representation of the general common solution compared to Mitra \cite{MR0320028}. They also derived a representation of the general Hermitian solution for $AXB=C$, provided the Hermitian solution exists. Wang et al. \cite{MR3575756} determined the least squares Hermitian solution with the minimum norm in the complex field using matrix and vector products. Zhang et al. \cite{MR3860589} developed a more efficient method compared to \cite{MR3575756} for finding the least squares Hermitian solution with the minimum norm by utilizing the real representation of a complex matrix. Yuan et al. \cite{MR3434526} studied the least squares L-structured solutions of quaternion linear matrix equations using the complex representation of quaternion matrices. However, limited research has been conducted on solving matrix equations in the reduced biquaternion domain.
	
A reduced biquaternion, a type of $4D$ hypercomplex number, shares with quaternions the characteristic of possessing a real part alongside three imaginary components. However, it distinguishes itself by offering notable advantages over its quaternion counterpart. One significant distinction is the commutative nature of multiplication, which streamlines fundamental operations like multiplication itself, SVD, and eigenvalue calculations \cite{pei2008eigenvalues}. This results in algorithms that are less complex and more expeditious in computational procedures. For instance, Pei et al. \cite{pei2004commutative,pei2008eigenvalues} and Gai \cite{gai2023theory} demonstrated that reduced biquaternions outperform conventional quaternions for image and digital signal processing. Despite the advantage of reduced biquaternions over quaternions, limited research has been dedicated to solving reduced biquaternion matrix equations (RBMEs). Recently, Yuan et al. \cite{MR4137050} studied the necessary and sufficient conditions for the existence of a Hermitian solution of the RBME \eqref{eq1.1} and provided a general formulation for the solution  by using the method of complex representation (CR) of reduced biquaternion matrices. However, this method involves complex matrices and complex arithmetic, leading to extensive computations. 

In the existing literature, several papers have employed the real representation method to solve matrix equations within the complex and quaternion domains. For instance, in \cite{MR3860589}, the authors utilized real representation of complex matrices to address the complex matrix equation $(AXB, CXD) = (E, F)$. In \cite{MR3537293,MR3406901, MR4486494}, similar approaches were taken using real representation of quaternion matrices to tackle quaternion matrix equations. These studies clearly demonstrate the superior efficiency of the real representation (RR) method compared to the complex representation (CR) method when dealing with matrix equations in the complex and quaternion domains. However, the real representation method has not yet been explored in the context of solving reduced biquaternion matrix equations. Motivated by insights from prior research, this paper introduces a more efficient approach for finding the Hermitian solution to RBME \eqref{eq1.1}. Here are the highlights of the work presented in this paper: 
	\begin{itemize}
		\item 	We explore the least squares Hermitian solution with the minimum norm of the RBME \eqref{eq1.1} using the method of real representation  of reduced biquaternion matrices. The notable advantage of this method lies in its exclusive use of real matrices and real arithmetic operations, resulting in enhanced efficiency. 
	    \item We establish the necessary and sufficient conditions for the existence and uniqueness of the Hermitian solution to RBME \eqref{eq1.1} and provide a general form of the solution. 
	    \item  In light of complex matrix equations being a special case of reduced biquaternion matrix equations, we utilize our developed method to find the Hermitian solution for matrix equation \eqref{eq1.1} over the complex field.
	    \item 	We conduct a comparative analysis between the CR and RR methods to study the Hermitian solution of RBME \eqref{eq1.1}.
	   \end{itemize}
Furthermore, this paper explores the application of the proposed framework in solving inverse eigenvalue problems. Inverse eigenvalue problems often involve the reconstruction of structured matrices based on given spectral data. When the available spectral data contains only partial information about the eigenpairs, this problem is termed as a partially described inverse eigenvalue problem (PDIEP). In PDIEP, two crucial aspects come to the forefront: the theory of solvability and the development of numerical solution methodologies (as detailed in \cite{chu2005inverse} and related references). In terms of solvability, a significant challenge has been to establish the necessary or sufficient conditions for solvability of PDIEP. On the other hand, numerical solution methods aim to construct matrices in a numerically stable manner when the given spectral data is feasible. In this paper, by leveraging our developed framework, we successfully introduce a numerical solution methodology for PDIEP \cite[Problem 5.1]{chu2005inverse}, which requires the construction of the Hermitian matrix from an eigenpair set.

The manuscript is organized as follows. In Section \ref{sec2}, notation and preliminary results are presented. Section \ref{sec3} outlines a framework for solving constrained RBME. Section \ref{sec4} highlights the application of our developed framework to solve PDIEP. Section \ref{sec5} offers numerical verification of our developed findings.

\section{Notation and preliminaries}\label{sec2}
	\subsection{Notation}
		Throughout this paper, $\QR$ denotes the set of all reduced biquaternions. $\R^{m \times n}$, $\C^{m \times n}$, and $\QR^{m \times n}$ represent the sets of all $m \times n$ real, complex, and reduced biquaternion matrices, respectively. We also denote $\SR^{n \times n}$, $\ASR^{n \times n}$, $\HC^{n \times n}$, and $\HQR^{n \times n}$ as the sets of all $n \times n$ real symmetric , real anti-symmetric, complex Hermitian, and reduced biquaternion Hermitian matrices, respectively. For a diagonal matrix $A=(a_{ij}) \in \QR^{n \times n}$, we denote it as $\diag(\alpha_1,\alpha_2,\ldots,\alpha_n)$, where $a_{ij}=0$ whenever $i \neq j$ and $a_{ii}=\alpha_i$ for $i=1,\ldots,n$. For $A \in \C^{m \times n}$, the notations $A^{+}, A^T, \Re(A)$, and $\Im(A)$ stand for the Moore-Penrose generalized inverse, transpose, real part, and imaginary part of $A$, respectively. $I_{n}$ represents the identity matrix of order $n$. For $i=1, 2, \ldots, n$, $e_i$ denotes the $i^{th}$ column of the identity matrix $I_n$. $0$ denotes the zero matrix of suitable size. $A \otimes B= (a_{ij}B)$ represents the Kronecker product of matrices $A \; \mbox{and} \; B$. The symbol $\norm{\cdot}_F$ represents the Frobenius norm. $\norm{\cdot}_2$ represents the $2$-norm or Euclidean norm. For $A \in \QR^{m \times n_1}$ and $B \in \QR^{m \times n_2}$, the notation $\left[A, B\right]$ represents the matrix $\begin{bmatrix}A & B\end{bmatrix} \in \QR^{m \times (n_1 + n_2)}$.
		
			Matlab command $randn(m,n)$ creates an $m \times n$ codistributed matrix of normally distributed random numbers whose every element is between $0$ and $1$. $rand(m,n)$ returns an $m \times n$ matrix of uniformly distributed random numbers. $ones(m,n)$ and $zeros(m,n)$ returns an $m \times n$ matrix of ones and zeros, respectively. $toeplitz(1:n)$ creates a Toeplitz matrix whose first row and first column is $[1, 2, \ldots, n]$. $eye(n)$ generates an identity matrix of size $n \times n$. Let $A$ be any matrix of size $n \times n$. $triu(A)$ returns the upper triangular part of a matrix $A$, while setting all the elements below the main diagonal to zero. $triu(A,k)$ returns the elements on and above the $k^{th}$ diagonal of $A$, while setting all the elements below it as zero. We use the following abbreviations throughout this paper:\\
		RBME : reduced biquaternion matrix equation,  CME : complex matrix equation, CR : complex representation, RR : real representation.

	\subsection{Preliminaries}
	A reduced biquaternion can be uniquely expressed as $a= a_{0}+ a_{1}\,\i+ a_{2}\,\j+ a_{3}\,\k$, where $a_{i} \in \R$ for $i= 0, 1, 2, 3$, and $\i^2=  \k^2= -1, \; \j^2= 1$,
	$\i\j= \j\i= \k, \; \j\k= \k\j= \i, \; \k\i= \i\k= -\j$. The norm of $a$ is $\norm{a}= \sqrt{a_0^2+ a_1^2+ a_2^2+ a_3^2}$. The Frobenius norm for $A= (a_{ij}) \in \QR^{m \times n}$ is defined as follows:
	\begin{equation}\label{eq2.1}
	\norm{A}_F= \sqrt{\sum_{i= 1}^{m} \sum_{j= 1}^{n} \norm{a_{ij}}^{2}}.
\end{equation}
	Let $A=A_0+ A_1\i+ A_2\j+ A_3\k \in \QR^{m \times n}$, where $A_t \in \R^{m \times n}$ for $t=0,1,2,3$. The real representation of matrix $A$, denoted as $A^R$, is defined as follows:
	\begin{equation}\label{eq2.2}
		A^R= \begin{bmatrix}
			A_0  &  -A_1  &  A_2  &  -A_3 \\
			A_1   &   A_0  &  A_3  &   A_2  \\
			A_2   &  -A_3 & A_0  &  -A_1  \\
			A_3  &  A_2  &  A_1  &  A_0 
		\end{bmatrix}.
	\end{equation}
Let $A_r^R$ denote the first block row of the block matrix $A^R$, i.e., $A_r^R= 
	[A_0, -A_1, A_2, -A_3].$ We have
	\begin{equation}\label{eq2.3}
	\norm{A}_F= \frac{1}{2} \norm{A^R}_F = \norm{A_r^R}_F.
\end{equation}
For matrix $A= (a_{ij}) \in \QR^{m \times n}$, let $a_{j}=\left[a_{1j}, a_{2j}, \ldots, a_{mj}\right] \; \mbox{for} \; j= 1,  2, \ldots, n$. We have $\vec(A)= \left[a_1, a_2, \ldots, a_n\right]^T$. Clearly, the operator $\vec(A)$ is linear, which means that for $A$, $B \in \QR^{m \times n}$, and $\alpha \in \R$, we have
$\vec(A+B)= \vec(A)+ \vec(B) \,\mbox{and}\,\vec(\alpha A)= \alpha \vec(A).$ Also, we have
\begin{equation}	\label{eq2.4}
	\vec(A_r^R) = \begin{bmatrix}
		\vec(A_0)\\
		-\vec(A_1) \\
		\vec(A_2) \\
		-\vec(A_3)
	\end{bmatrix}.
\end{equation}
For $A \in \C^{m \times n}$, $B \in \C^{ n \times s}$,  and $C \in \C^{s \times t}$, it is well known that $\vec(ABC) = \left(C^T \otimes A\right)\vec(B)$.
Now, we present two key lemmas. The subsequent lemma can be easily deduced from the structures of $A^R$ and $A_r^R$.
\begin{lemma}\label{lem2.1}
	Let $A, B \in \QR^{m \times n}$, $C \in \QR^{n \times p}$, and $\alpha \in \R$. Then the following properties hold.
		\begin{enumerate}[label=\textup{(\roman*)}, noitemsep,nolistsep]
		\item $A=B \iff A^R=B^R \iff A_r^R=B_r^R$.
		\item $(A+B)^R = A^R + B^R$, $(\alpha A)^R=\alpha A^R$, $(AC)^R = A^RC^R$.
		\item $(A+B)_r^R=A_r^R+B_r^R$, $(\alpha A)_r^R = \alpha A_r^R$, $(AC)_r^R=A_r^R C^R$.
	\end{enumerate}
\end{lemma}

\begin{lemma}\label{lem2.2}
		Let $X=X_0+X_1\i+X_2\j+X_ 3\k \in \QR^{n \times n}$. Then $\vec(X^R) =\mathcal{J}\vec(X_r^R)$. We have
		\begin{equation}\label{eq2.5}
			\mathcal{J}=
	\begin{bmatrix}
		\mathcal{J}_0 & 	-\mathcal{J}_1 & 	\mathcal{J}_2 & 	-\mathcal{J}_3\\
		\mathcal{J}_1 & 	\mathcal{J}_0 & 	\mathcal{J}_3 & 	\mathcal{J}_2\\
		\mathcal{J}_2 & 	-\mathcal{J}_3 & 	\mathcal{J}_0 & 	-\mathcal{J}_1\\
	    \mathcal{J}_3 & 	\mathcal{J}_2 & 	\mathcal{J}_1 & 	\mathcal{J}_0
	\end{bmatrix},
\end{equation} 
 where  $\mathcal{J}_0=\begin{bmatrix}
	\mathcal{J}_{01} \\
	\mathcal{J}_{02} \\
	\vdots \\
	\mathcal{J}_{0n} \\
\end{bmatrix}$, 
$\mathcal{J}_1=\begin{bmatrix}
\mathcal{J}_{11} \\
\mathcal{J}_{12} \\
\vdots \\
\mathcal{J}_{1n} \\
\end{bmatrix}$,
$\mathcal{J}_2=\begin{bmatrix}
\mathcal{J}_{21} \\
\mathcal{J}_{22} \\
\vdots \\
\mathcal{J}_{2n} \\
\end{bmatrix}$,
$\mathcal{J}_3=\begin{bmatrix}
\mathcal{J}_{31} \\
\mathcal{J}_{32} \\
\vdots \\
\mathcal{J}_{3n} \\
\end{bmatrix}$. $\mathcal{J}_{ij}$ is a block matrix of size $4 \times n$ with $I_n$ at the $(i+1,j)^{th}$ position, and the rest of the entries are zero matrices of size $n$, where $i=0,1,2,3$ and $j=1,2,\ldots,n$.
\end{lemma}
The above lemma can be easily derived through direct computation; thus, we omit the proof. To enhance our understanding of the above lemma, we will examine it in the context of $n = 2$. In this scenario, we have

$\mathcal{J}_0=\begin{bmatrix}
	\mathcal{J}_{01} \\
	\mathcal{J}_{02}
\end{bmatrix}=\begin{bmatrix}
		I_2       &   0        \\
		0          &   0         \\
		0          &    0         \\
		0          &    0          \\
	 	\hdashline
	 	0          &  I_2         \\
		0          &   0            \\
	    0          &    0           \\
		0          &    0           
	\end{bmatrix}$, 
	$\mathcal{J}_1=\begin{bmatrix}
		\mathcal{J}_{11} \\
		\mathcal{J}_{12}
	\end{bmatrix}=\begin{bmatrix}
			0       &   0           \\
			I_2          &   0      \\
			0          &    0         \\
			0          &    0          \\
				\hdashline
			0          &  0            \\
			0          &   I_2         \\
			0          &    0           \\
			0          &    0            
		\end{bmatrix}$, 
		$\mathcal{J}_2=\begin{bmatrix}
			\mathcal{J}_{21} \\
			\mathcal{J}_{22}
		\end{bmatrix}=\begin{bmatrix}
				0       &   0              \\
				0         &   0             \\
				I_2       &    0            \\
				0          &    0            \\
					\hdashline
				0          &  0             \\
				0          &   0             \\
				0          &    I_2           \\
				0          &    0            
			\end{bmatrix}$, 
			$\mathcal{J}_3=\begin{bmatrix}
				\mathcal{J}_{31} \\
				\mathcal{J}_{32}
			\end{bmatrix}=\begin{bmatrix}
					0       &   0            \\
					0         &   0           \\
					0       &    0            \\
				   I_2          &    0        \\
						\hdashline
					0          &  0            \\
					0          &   0            \\
					0          &    0            \\
					0          &   I_2         
				\end{bmatrix}.$\\
			
Now, we recall some fundamental results that are pivotal in establishing the main findings of this paper.
\begin{definition}
	For $X=(x_{ij}) \in \R^{n \times n}$, let $\alpha_1=\left[x_{11}, x_{21}, \ldots, x_{n1}\right]$, $\alpha_2=\left[x_{22},x_{32},\ldots,x_{n2}\right]$, $\ldots$, $\alpha_{n-1}=\left[x_{(n-1)(n-1)},x_{n(n-1)}\right]$, $\alpha_n = x_{nn}$, and denote by $\vec_S(X)$ the following vector:
	\begin{equation}\label{eqs}
      \vec_S(X)=\left[\alpha_1,\alpha_2,\ldots,\alpha_{n-1},\alpha_n\right]^T \in \R^{\frac{n(n+1)}{2}}.
	\end{equation}
\end{definition}

\begin{definition}
	For $X=(x_{ij}) \in \R^{n \times n}$, let $\beta_1=\left[x_{21},x_{31}, \ldots, x_{n1}\right] $, $\beta_2=\left[x_{32},x_{42}, \ldots,x_{n2}\right]$, $\ldots$, $\beta_{n-2}=\left[x_{(n-1)(n-2)},x_{n(n-2)}\right]$, $\beta_{n-1} = x_{n(n-1)}$, and denote by $\vec_A(X)$ the following vector:
	\begin{equation}\label{eqa}
		\vec_A(X)=\left[\beta_1,\beta_2,\ldots,\beta_{n-2},\beta_{n-1}\right]^T \in \R^{\frac{n(n-1)}{2}}.
	\end{equation}
\end{definition}
\begin{lemma}\label{lem2.5}\cite{MR3760735}
	If $X \in \R^{n \times n}$, then
	\begin{enumerate}[label=\textup{(\roman*)}, noitemsep,nolistsep]
		\item $X \in \SR^{n \times n} \iff \vec(X) = K_S\vec_S(X)$, where $\vec_S(X)$ is of the form \eqref{eqs}, and the matrix $K_S \in \R^{n^2 \times \frac{n(n+1)}{2}}$ is represented as
		\begin{equation*}
	K_S=	\begin{blockarray}{ccccccccccccccc}
			\begin{block}{[ccccccccccccccc]}
			e_{1}&e_{2}&e_{3}&\cdots&e_{n-1}&e_{n}&0&0 &\cdots&0& 0&\cdots& 0&0& 0 \\
			0&e_{1} &0&\cdots &0 &0&e_{2}& e_{3}& \cdots&e_{n-1}&e_{n}&\cdots&0&0&0 \\
			0&0&e_{1}&\cdots&0&0&0&e_{2}& \cdots &0&0& \cdots& 0 &  0   &  0   \\
	        \vdots &\vdots&\vdots& &\vdots &\vdots &\vdots& \vdots&&\vdots&\vdots&             &\vdots &\vdots       &\vdots \\
	        0&0&0&\cdots&e_{1}&0&0&0&\cdots&e_{2}&0&\cdots&e_{n-1}&e_{n}&0\\
	        0&0&0&\cdots&0&e_{1}&0&0&\cdots&0&e_{2}&\cdots&0&e_{n-1}&e_{n}\\
	        \end{block}
		\end{blockarray}.
	\end{equation*}
\item $X \in \ASR^{n \times n} \iff \vec(X) = K_A\vec_A(X)$, where $\vec_A(X)$ is of the form \eqref{eqa}, and the matrix $K_A \in \R^{n^2 \times \frac{n(n-1)}{2}}$ is represented as
\begin{equation*}
	K_A=	\begin{blockarray}{ccccccccccc}
		\begin{block}{[ccccccccccc]}
			e_{2} & e_{3} & \cdots & e_{n-1}  & e_{n} & 0 & \cdots & 0 & 0 & \cdots & 0 \\
			-e_{1} & 0 & \cdots & 0 & 0 &e_{3} & \cdots & e_{n-1} & e_{n} & \cdots & 0 \\
			0 & -e_{1} & \cdots & 0 & 0 & -e_{2} & \cdots&0 & 0 & \cdots & 0 \\
			\vdots & \vdots & & \vdots & \vdots & \vdots & & \vdots & \vdots & & \vdots\\
			0 & 0 & \cdots & -e_{1} & 0 & 0 & \cdots & -e_{2} & 0 & \cdots & e_{n} \\
			0 & 0& \cdots & 0 & -e_{1} & 0 & \cdots & 0 & -e_{2} & \cdots & -e_{n-1} \\
		\end{block}
	\end{blockarray}.
\end{equation*}
	\end{enumerate}
\end{lemma}

For any $X=X_0+X_1\i+X_2\j+X_3\k \in \HQR^{n \times n} $, it is easy to see that
\begin{equation*}
	X \in \HQR^{n \times n} \iff X_0^T=X_0, \; X_1^T=-X_1, \;  X_2^T=-X_2, \; X_3^T=-X_3.
\end{equation*}
Set 
\begin{equation}\label{eq2.6}
	\mathcal{Q}=\begin{bmatrix}
		K_S   &  0    &  0     &  0\\
		0       &-K_A  &0       &  0\\
		0       & 0      & K_A  &  0\\
		0       & 0      & 0       &-K_A
	\end{bmatrix}.
\end{equation}
\begin{lemma}\label{lem2.6}
	If $X=X_0+X_1\i+X_2\j+X_3\k \in \QR^{n \times n}$, then $$X \in \HQR^{n \times n} \iff \vec(X_r^R)= \mathcal{Q}
\begin{bmatrix}
	\vec_S(X_0) \\
	\vec_A(X_1) \\
	\vec_A(X_2) \\
	\vec_A(X_3) 
\end{bmatrix}.$$
\end{lemma}
\begin{proof}
For $X \in \HQR^{n \times n}$, we have $ X_0 \in \SR^{n \times n}$ and $X_1$, $X_2$, $X_3 \in \ASR^{n \times n}$. By applying Lemma \ref{lem2.5}, we obtain $\vec(X_0) = K_S\vec_S(X_0)$, $\vec(X_1) = K_A\vec_A(X_1)$, $\vec(X_2) = K_A\vec_A(X_2)$, and $\vec(X_3) = K_A\vec_A(X_3)$. By utilizing \eqref{eq2.4} and \eqref{eq2.6}, we achieve the desired result.
\end{proof}
Additionally, we need the following lemma for developing the main results.
		\begin{lemma}\label{lem2.7} \cite{golub2013matrix} Consider the matrix equation of the form $Ax=b$, where $A \in \R^{m \times n}$ and $b \in \R^{m}$. The following results hold:
		\begin{enumerate}[label=\textup{(\roman*)}, noitemsep,nolistsep]
			\item The matrix equation has a solution $x$ if and only if 
			$AA^+b = b$. In this case, the general solution is $x= A^+b+\left(I-A^+A\right)y,$ where $y \in \R^n$ is an arbitrary vector. Furthermore, if the consistency condition is satisfied, then the matrix equation has a unique solution if and only if $\rank(A) = n$. In this case, the unique solution is $x= A^+b$.
			\item The least squares solutions of the matrix equation can be expressed as
			$x= A^+b+\left(I-A^+A\right)y,$ where $y \in \R^n$ is an arbitrary vector, and the least squares solution with the least norm is $x=A^+b $.
		\end{enumerate}
	\end{lemma}
\section{Framework for solving constrained RBME}\label{sec3}

In this section, we initially investigate the least squares Hermitian solution with the least norm of RBME \eqref{eq1.1} using the RR method. The problem is formulated as follows:
\begin{problem} \label{pb1}
	Given matrices $A,C \in \QR^{m \times n}$, $B,D \in \QR^{n \times s}$, and $E,F \in \QR^{m \times s}$, let
	\begin{equation*}
		\mathcal{H}_{LE}= \left\{ X \; | \; X \in \HQR^{n \times n}, \;  \norm{\left(AXB-E, CXD-F\right)}_F= \min_{\widetilde{X} \in \HQR^{n \times n}} \norm{\left(A\widetilde{X}B-E, C\widetilde{X}D-F\right)}_F \right\}.
	\end{equation*}
	Then, find $X_H \in \mathcal{H}_{LE}$ such that
	$$\norm{X_H}_F= \underset{X \in \mathcal{H}_{LE}}{\min}\norm{X}_F.$$
\end{problem}
Next, by utilizing the RR method, we derive the necessary and sufficient conditions for the existence of the Hermitian solution for RBME \eqref{eq1.1}, along with a general formulation for the solution if the consistency criterion is met. We also establish the conditions for a unique solution and, in such cases, provide a general form of the solution. The RR method involves transforming the constrained RBME \eqref{eq1.1} into an equivalent unconstrained real linear system. 

Following that, we present a concise overview of the CR method, which is employed to find the Hermitian solution for RBME  \eqref{eq1.1} as documented in \cite{MR4137050}. Additionally, we conduct a comparative analysis between the RR and CR methods. 

Before proceeding, we introduce some notations that will be used in the subsequent results. Let
\begin{equation}\label{eq3.1}
\mathcal{P} = 	\begin{bmatrix}
		(B^R)^T \otimes A_r^R \\
		(D^R)^T \otimes C_r^R
	\end{bmatrix}, \; \;  \mathcal{R}=\diag(K_S, K_A, K_A, K_A).
\end{equation}
\begin{theorem}\label{thm1}
Given matrices $A,C \in \QR^{m \times n}$, $B,D \in \QR^{n \times s}$, and $E,F \in \QR^{m \times s}$, let $X=X_0+X_1\i+X_2\j+X_3\k  \in \QR^{n \times n}$. Let $\mathcal{P}$ and $\mathcal{R}$ be of the form \eqref{eq3.1}, and let $\mathcal{J}$ and $\mathcal{Q}$ be of the form \eqref{eq2.5} and \eqref{eq2.6}, respectively. Then, the set $\mathcal{H}_{LE}$ of Problem \ref{pb1} can be expressed as
 \begin{equation}\label{eq3.2}
 	\mathcal{H}_{LE}=\left\{X \; \left| \; \begin{bmatrix}
 			\vec(X_0) \\
 			\vec(X_1)  \\
 			\vec(X_2) \\
 			\vec(X_3)
 		\end{bmatrix}\right.=\mathcal{R}\left(\mathcal{P}\mathcal{J}\mathcal{Q}\right)^+
 	\begin{bmatrix}
 		\vec(E_r^R) \\
 		\vec(F_r^R)
 	\end{bmatrix} + \mathcal{R}\left(I_{(2n^2-n)}-\left(\mathcal{P}\mathcal{J}\mathcal{Q}\right)^+\left(\mathcal{P}\mathcal{J}\mathcal{Q}\right)\right)y
 	\right\},
 \end{equation}
	where $y \in \R^{(2n^2-n)}$ is an arbitrary vector. Furthermore, the unique solution $X_{H} \in \mathcal{H}_{LE}$ to Problem \ref{pb1} satisfies
\begin{equation}\label{eq3.3}
	 \begin{bmatrix}
		\vec(X_0) \\
		\vec(X_1)  \\
		\vec(X_2) \\
		\vec(X_3)
	\end{bmatrix}= \mathcal{R}\left(\mathcal{P}\mathcal{J}\mathcal{Q}\right)^+	\begin{bmatrix}
	\vec(E_r^R) \\
	\vec(F_r^R)
\end{bmatrix}. 
\end{equation}	
\end{theorem}
\begin{proof}
	By using \eqref{eq2.1}, \eqref{eq2.3}, and Lemma \ref{lem2.1}, we get
	\begin{eqnarray*}
		 \norm{\left(AXB-E, CXD-F\right)}_F^2 &=&  \norm{\left(AXB-E\right)}_F^2+ \norm{\left( CXD-F\right)}_F^2\\
		 &=& \norm{\left(AXB-E\right)_r^R}_F^2+ \norm{\left( CXD-F\right)_r^R}_F^2\\
		 &=& \norm{\left(A_r^R X^R B^R - E_r^R\right)}_F^2+ \norm{\left( C_r^R X^R D^R-F_r^R\right)}_F^2\\
		 &=&\norm{\vec\left(A_r^R X^R B^R - E_r^R\right)}_F^2+ \norm{\vec\left( C_r^R X^R D^R-F_r^R\right)}_F^2.
	\end{eqnarray*}
We have
\begin{eqnarray*}
	\vec\left(A_r^R X^R B^R - E_r^R\right) &=& \left( \left(B^R\right)^T \otimes A_r^R\right)\vec(X^R)-\vec(E_r^R), \\
	\vec\left(C_r^R X^R D^R - F_r^R\right) &=& \left( \left(D^R\right)^T \otimes C_r^R\right)\vec(X^R)-\vec(F_r^R). 
\end{eqnarray*}
By using Lemma \ref{lem2.2} and Lemma \ref{lem2.6}, we get
\begin{equation*}
	\vec(X^R) = \mathcal{J}\vec\left(X_r^R\right) =  \mathcal{J} \mathcal{Q}
	\begin{bmatrix}
		\vec_S(X_0) \\
		\vec_A(X_1) \\
		\vec_A(X_2) \\
		\vec_A(X_3) 		
	\end{bmatrix}.
 \end{equation*}
Thus,
\begin{eqnarray*}
	 \norm{\left(AXB-E, CXD-F\right)}_F^2 &=& \norm{
	 	\begin{bmatrix}
	 		\left(B^R\right)^T \otimes A_r^R \\
	 		\left(D^R\right)^T \otimes C_r^R 
	 \end{bmatrix}\vec\left(X^R\right)-
 \begin{bmatrix}
 	\vec(E_r^R) \\
 	\vec(F_r^R)
\end{bmatrix}}_F^2\\
  &=& \norm{\mathcal{P}\mathcal{J}\mathcal{Q}
  		\begin{bmatrix}
  		\vec_S(X_0) \\
  		\vec_A(X_1) \\
  		\vec_A(X_2) \\
  		\vec_A(X_3) 		
  \end{bmatrix}-
 \begin{bmatrix}
	\vec(E_r^R) \\
	\vec(F_r^R)
\end{bmatrix}}_F^2.
\end{eqnarray*}
Hence, Problem \ref{pb1} can be solved by finding the least squares solutions of the following unconstrained real matrix system: 
\begin{equation}\label{pr1}
	\mathcal{P}\mathcal{J}\mathcal{Q}	
	\begin{bmatrix}
		\vec_S(X_0) \\
		\vec_A(X_1) \\
		\vec_A(X_2) \\
		\vec_A(X_3) 		
	\end{bmatrix}= \begin{bmatrix}
	\vec(E_r^R) \\
	\vec(F_r^R)
\end{bmatrix}.
\end{equation}
By using the fact that $X \in \HQR^{n \times n} \iff X_0 \in \SR^{n \times n}, X_1  \in \ASR^{n \times n}, X_2  \in \ASR^{n \times n}, X_3 \in \ASR^{n \times n}$, and using Lemma \ref{lem2.5}, we have
\begin{equation}\label{pr2}
	\begin{bmatrix}
		\vec(X_0) \\
		\vec(X_1) \\
		\vec(X_2) \\
		\vec(X_3) 		
	\end{bmatrix}=\mathcal{R}
	\begin{bmatrix}
		\vec_S(X_0) \\
		\vec_A(X_1) \\
		\vec_A(X_2) \\
		\vec_A(X_3) 		
	\end{bmatrix}.
\end{equation}
By using Lemma \ref{lem2.7}, the least squares solutions of matrix equation \eqref{pr1} is given by
 \begin{equation}\label{pr3}
 \begin{bmatrix}
		\vec_S(X_0) \\
		\vec_A(X_1)  \\
		\vec_A(X_2) \\
		\vec_A(X_3)
	\end{bmatrix}=\left(\mathcal{P}\mathcal{J}\mathcal{Q}\right)^+
	\begin{bmatrix}
		\vec(E_r^R) \\
		\vec(F_r^R)
	\end{bmatrix} + \left(I_{(2n^2-n)}-\left(\mathcal{P}\mathcal{J}\mathcal{Q}\right)^+\left(\mathcal{P}\mathcal{J}\mathcal{Q}\right)\right)y,
\end{equation}
where $y \in \R^{(2n^2-n)}$ is an arbitrary vector. Hence, we get \eqref{eq3.2} by utilizing \eqref{pr2} and \eqref{pr3}.

In addition, by employing Lemma \ref{lem2.7}, the least squares solution with the least norm for matrix equation \eqref{pr1} is given by
 \begin{equation}\label{pr4}
	\begin{bmatrix}
		\vec_S(X_0) \\
		\vec_A(X_1)  \\
		\vec_A(X_2) \\
		\vec_A(X_3)
	\end{bmatrix}=\left(\mathcal{P}\mathcal{J}\mathcal{Q}\right)^+
	\begin{bmatrix}
		\vec(E_r^R) \\
		\vec(F_r^R)
	\end{bmatrix}.
\end{equation}
Hence, we get \eqref{eq3.3} by utilizing \eqref{pr2} and \eqref{pr4}. 
\end{proof}
\begin{theorem}\label{thm2}
Let $\mathcal{P}$ and $\mathcal{R}$ be in the form of \eqref{eq3.1}, and let $X=X_0+X_1\i+X_2\j+X_3\k \in \QR^{n \times n}$. Additionally, let $\mathcal{J}$ and $\mathcal{Q}$ be as in \eqref{eq2.5} and \eqref{eq2.6}, respectively. Then, the RBME \eqref{eq1.1} has a Hermitian solution $X\in \HQR^{n \times n}$ if and only if
	\begin{equation}\label{eq3.4}
	\left(I_{8ms}-\left(\mathcal{P}\mathcal{J}\mathcal{Q}\right)\left(\mathcal{P}\mathcal{J}\mathcal{Q}\right)^+\right)\begin{bmatrix}
		\vec(E_r^R) \\
		\vec(F_r^R)
	\end{bmatrix}=0.
	\end{equation}	
 In this case, the general solution $X \in  \HQR^{n \times n}$ satisfies
	\begin{equation}\label{eq3.5}
		\begin{bmatrix}
			\vec(X_0) \\
			\vec(X_1) \\
			\vec(X_2) \\
			\vec(X_3) 
		\end{bmatrix} = \mathcal{R}\left(\mathcal{P}\mathcal{J}\mathcal{Q}\right)^+\begin{bmatrix}
		\vec(E_r^R) \\
		\vec(F_r^R)
	\end{bmatrix}+\mathcal{R}\left(I_{(2n^2-n)}-\left(\mathcal{P}\mathcal{J}\mathcal{Q}\right)^+\left(\mathcal{P}\mathcal{J}\mathcal{Q}\right)\right)y,
	\end{equation}
	where $y \in \R^{(2n^2-n)}$ is an arbitrary vector. Furthermore, if \eqref{eq3.4} holds, then the RBME \eqref{eq1.1} has a unique solution $X \in \HQR^{n \times n}$ if and only if 
	\begin{equation}\label{eq3.6}
		\rank \left(\mathcal{P}\mathcal{J}\mathcal{Q}\right)=2n^2-n.
	\end{equation}
	In this case, the unique solution $X \in  \HQR^{n \times n}$ satisfies
	\begin{equation}\label{eq3.7}
			\begin{bmatrix}
			\vec(X_0) \\
			\vec(X_1) \\
			\vec(X_2) \\
			\vec(X_3) 
		\end{bmatrix} = \mathcal{R}\left(\mathcal{P}\mathcal{J}\mathcal{Q}\right)^+\begin{bmatrix}
			\vec(E_r^R) \\
			\vec(F_r^R)
		\end{bmatrix}. 
	\end{equation}
\end{theorem}
\begin{proof}
	By using Lemmas \ref{lem2.1}, \ref{lem2.2}, and \ref{lem2.6}, we have
	\begin{eqnarray*}
		\left(AXB,CXD\right) = \left(E,F\right) &\iff& 	\left(\left(AXB\right)_r^R,\left(CXD\right)_r^R\right) = \left(E_r^R,F_r^R\right) \\
		& \iff & \left(A_r^R X^R B^R, C_r^R X^R D^R\right) = \left(E_r^R, F_r^R\right) \\
		& \iff & \left(\vec\left(A_r^R X^R B^R\right), \vec\left(C_r^R X^R D^R\right) \right) =
		\left(\vec\left(E_r^R\right), \vec\left(F_r^R\right)\right) \\
		& \iff & \begin{bmatrix}
			\vec\left(A_r^R X^R B^R\right) \\
			\vec\left(C_r^R X^R D^R\right) 
		\end{bmatrix} = 		
	\begin{bmatrix}
		\vec(E_r^R) \\
		\vec(F_r^R)
	\end{bmatrix}\\
& \iff & \begin{bmatrix}
				\left(B^R\right)^T \otimes A_r^R \\
				\left(D^R\right)^T \otimes C_r^R 
			\end{bmatrix}\vec\left(X^R\right)=
			\begin{bmatrix}
				\vec(E_r^R) \\
				\vec(F_r^R)
			\end{bmatrix} \\
		& \iff & \mathcal{P}\mathcal{J}\mathcal{Q}
			\begin{bmatrix}
			\vec_S(X_0) \\
			\vec_A(X_1) \\
			\vec_A(X_2) \\
			\vec_A(X_3) 
		\end{bmatrix} = 	
	\begin{bmatrix}
		\vec(E_r^R) \\
		\vec(F_r^R)
	\end{bmatrix}.
	\end{eqnarray*}
Hence, solving the constrained RBME \eqref{eq1.1} is equivalent to solving the following unconstrained real matrix system:
\begin{equation}\label{pr5}
	\mathcal{P}\mathcal{J}\mathcal{Q}	
	\begin{bmatrix}
		\vec_S(X_0) \\
		\vec_A(X_1) \\
		\vec_A(X_2) \\
		\vec_A(X_3) 		
	\end{bmatrix}= \begin{bmatrix}
		\vec(E_r^R) \\
		\vec(F_r^R)
	\end{bmatrix}.
\end{equation}
By using the fact that $X \in \HQR^{n \times n} \iff X_0 \in \SR^{n \times n}, X_1  \in \ASR^{n \times n}, X_2  \in \ASR^{n \times n}, X_3 \in \ASR^{n \times n}$, and using Lemma \ref{lem2.5}, we have
\begin{equation}\label{pr6}
	\begin{bmatrix}
		\vec(X_0) \\
		\vec(X_1) \\
		\vec(X_2) \\
		\vec(X_3) 		
	\end{bmatrix}=\mathcal{R}
	\begin{bmatrix}
		\vec_S(X_0) \\
		\vec_A(X_1) \\
		\vec_A(X_2) \\
		\vec_A(X_3) 		
	\end{bmatrix}.
\end{equation}
By leveraging Lemma \ref{lem2.7}, we obtain the consistency condition \eqref{eq3.4}. This condition establishes the criterion that ensures a solution to matrix equation \eqref{pr5} and a Hermitian solution for RBME \eqref{eq1.1}. In this case, the solution to matrix equation \eqref{pr5} can be expressed as
	\begin{equation}\label{pr7}
	\begin{bmatrix}
		\vec_S(X_0) \\
		\vec_A(X_1) \\
		\vec_A(X_2) \\
		\vec_A(X_3) 
	\end{bmatrix} = \left(\mathcal{P}\mathcal{J}\mathcal{Q}\right)^+\begin{bmatrix}
		\vec(E_r^R) \\
		\vec(F_r^R)
	\end{bmatrix}+\left(I_{(2n^2-n)}-\left(\mathcal{P}\mathcal{J}\mathcal{Q}\right)^+\left(\mathcal{P}\mathcal{J}\mathcal{Q}\right)\right)y,
\end{equation}
where $y \in \R^{(2n^2-n)}$ is an arbitrary vector. By applying \eqref{pr6} and \eqref{pr7}, we get that the Hermitian solution to RBME \eqref{eq1.1} satisfies \eqref{eq3.5}.

Furthermore, if condition \eqref{eq3.4} is satisfied, applying Lemma \ref{lem2.7} leads to the derivation of the condition \eqref{eq3.6}. This condition signifies the criterion for ensuring a unique solution to matrix equation \eqref{pr5} and a unique Hermitian solution to RBME \eqref{eq1.1}. In this case, the unique solution to matrix equation \eqref{pr5} is given by
	\begin{equation}\label{pr8}
	\begin{bmatrix}
		\vec_S(X_0) \\
		\vec_A(X_1) \\
		\vec_A(X_2) \\
		\vec_A(X_3) 
	\end{bmatrix} = \left(\mathcal{P}\mathcal{J}\mathcal{Q}\right)^+\begin{bmatrix}
		\vec(E_r^R) \\
		\vec(F_r^R)
	\end{bmatrix}.
\end{equation}
By using \eqref{pr6} and \eqref{pr8}, we get that the unique Hermitian solution to RBME \eqref{eq1.1} satisfies \eqref{eq3.7}.
\end{proof}
In \cite{MR4137050}, the authors investigated the Hermitian solution of RBME \eqref{eq1.1} using the CR method. They successfully established the necessary and sufficient conditions for the existence of a solution, along with providing a general expression for it. In the remainder of this section, we will present the findings from \cite{MR4137050}.\\
A reduced biquaternion matrix $A \in \QR^{m \times n}$ is uniquely expressed as $A=A_1+ A_2\j$, where $A_1$, $A_2 \in \C^{m \times n}$. We have $\Psi_A= \left[A_1, A_2\right]  \in \C^{m \times 2n}$ and $\vec(\Psi_A)=\begin{bmatrix}
	\vec(A_1) \\
	\vec(A_2)
\end{bmatrix}.$ The complex representation of matrix $A$, denoted as $h(A)$, is defined as $$
	h(A)=\begin{bmatrix}
		A_1 & A_2\\
		A_2 & A_1                    
	\end{bmatrix}.                         
$$
Given matrices $A=A_1+A_2\j \in \QR^{m \times n}$, $B=B_1+B_2\j \in \QR^{n \times s}$, $C=C_1+C_2\j \in \QR^{m \times n}$, $D=D_1+D_2\j \in \QR^{n \times s}$, $E=E_1+E_2\j \in \QR^{m \times s}$, and $F=F_1+F_2\j \in \QR^{m \times s}$. We have
\begin{equation*}
	\left(AXB, CXD\right)=\left(E, F\right) \iff \left(\vec(\Psi_{AXB}), \vec(\Psi_{CXD})\right) = \left(\vec(\Psi_{E}), \vec(\Psi_{F})\right).
\end{equation*}
Denote

$\mathcal{R}=\diag(K_S, K_A, K_A, K_A)$, \; $\mathcal{U}=\begin{bmatrix}
	K_S & \i K_A & 0 & 0 \\
	0 & 0 & K_A & \i K_A
\end{bmatrix}$, \; $ x= \begin{bmatrix}
\vec_S(\Re(X_1)) \\
\vec_A(\Im(X_1)) \\
\vec_A(\Re(X_2)) \\
\vec_A(\Im(X_2)) 
\end{bmatrix}, 
$\\

$M =h\left(\left(B_1^T \otimes A_1 + B_2^T \otimes A_2\right) + \left(B_1^T \otimes A_2 + B_2^T \otimes A_1\right)\j\right)$, \\

$N=h\left(\left(D_1^T \otimes C_1 + D_2^T \otimes C_2\right) + \left(D_1^T \otimes C_2 + D_2^T \otimes C_1\right)\j\right)$. \\

\noin
According to \cite[Lemma~2.3]{MR4137050}, for $X \in \HQR^{n \times n}$, the authors of \cite{MR4137050} obtained
\begin{eqnarray*}
	\vec(\Psi_{AXB}) &=& h\left(\left(B_1^T \otimes A_1 + B_2^T \otimes A_2\right) + \left(B_1^T \otimes A_2 + B_2^T \otimes A_1\right)\j\right)\mathcal{U}x = M\mathcal{U}x, \\
		\vec(\Psi_{CXD}) &=& h\left(\left(D_1^T \otimes C_1 + D_2^T \otimes C_2\right) + \left(D_1^T \otimes C_2 + D_2^T \otimes C_1\right)\j\right)\mathcal{U}x = N\mathcal{U}x.
\end{eqnarray*}
Thus, the matrix equation $(AXB, CXD) = (E,F)$ for $X \in \HQR^{n \times n}$ is equivalent to
\begin{equation*}
	\begin{bmatrix}
		M \\
		N
	\end{bmatrix}\mathcal{U}x=\begin{bmatrix}
	\vec(E_1) \\
	\vec(E_2)\\
	\vec(F_1) \\
	\vec(F_2)
\end{bmatrix}.
\end{equation*}
Denote
\begin{equation}\label{cr1}
	Q=\begin{bmatrix}
M \\
N
\end{bmatrix}\mathcal{U}, \; Q_1 = \Re(Q), \; Q_2 = \Im(Q),  \; e=\begin{bmatrix}
\vec(\Re(E_1)) \\
\vec(\Re(E_2)) \\
\vec(\Re(F_1)) \\		
\vec(\Re(F_2)) \\
\vec(\Im(E_1)) \\		
\vec(\Im(E_2)) \\	
\vec(\Im(F_1)) \\		
\vec(\Im(F_2)) \\		   	 
\end{bmatrix}.
\end{equation}
We have
\begin{equation}\label{cr2}
	\begin{bmatrix}
		Q_1   \\
		Q_2
	\end{bmatrix}^+= \left[Q_1^+ - H^T Q_2Q_1^+, H^T\right], \;
	\begin{bmatrix}
		Q_1 \\
		Q_2
	\end{bmatrix}^+
	\begin{bmatrix}
		Q_1 \\
		Q_2
	\end{bmatrix}= Q_1^+Q_1+ R R^+, 
\end{equation}
where
\begin{equation}\label{cr3}
	\left.\begin{aligned}
		H &=  R^+ + \left(I-R^+R\right)ZQ_2Q_1^+Q_1^{+T}\left(I- Q_2^TR^+\right), \; \; R=  \left(I- Q_1^+ Q_1\right)Q_2^T	,\\
		Z    & = \left(I+ \left(I- R^+R\right)Q_2Q_1^+Q_1^{+T}Q_2^T\left(I- R^+R\right)\right)^{-1}. 
	\end{aligned}\right\}
\end{equation}
As per \cite{MR4137050}, solving the constrained RBME \eqref{eq1.1} is equivalent to solving the following unconstrained real matrix system:
\begin{equation*}
	\begin{bmatrix}
		Q_1 \\
		Q_2
	\end{bmatrix}x=e.
\end{equation*}
According to \cite[Theorem~3.1]{MR4137050}, the RBME \eqref{eq1.1} has a Hermitian solution $X=X_0+X_1\i+X_2\j+X_3\k \in \HQR^{n \times n}$ if and only if
\begin{equation}\label{eq3.10}
	\left(I_{8ms}-\begin{bmatrix}
		Q_1 \\
		Q_2
	\end{bmatrix}\begin{bmatrix}
	Q_1 \\
	Q_2
\end{bmatrix}^+\right)e=0.
\end{equation}	
In this case, the general solution $X \in  \HQR^{n \times n}$ satisfies
\begin{equation}\label{eq3.11}
	\begin{bmatrix}
		\vec(\Re(X_1)) \\
		\vec(\Im(X_1)) \\
		\vec(\Re(X_2)) \\
		\vec(\Im(X_2)) 
	\end{bmatrix} = \mathcal{R}\left[Q_1^+ - H^T Q_2Q_1^+, H^T\right]e + \mathcal{R}\left(I_{(2n^2-n)}-Q_1^+Q_1- R R^+\right)y,
\end{equation}
where $y \in \R^{(2n^2-n)}$ is an arbitrary vector. Furthermore, if the consistency condition is satisfied, then the RBME \eqref{eq1.1} has a unique solution $X \in \HQR^{n \times n}$ if and only if 
\begin{equation}\label{eq3.12}
	\rank \left(\begin{bmatrix}
		Q_1 \\
		Q_2
	\end{bmatrix}\right)=2n^2-n.
\end{equation}
In this case, the unique solution $X \in  \HQR^{n \times n}$ satisfies
\begin{equation}\label{eq3.13}
	\begin{bmatrix}
		\vec(\Re(X_1)) \\
		\vec(\Im(X_1)) \\
		\vec(\Re(X_2)) \\
		\vec(\Im(X_2)) 
	\end{bmatrix} = \mathcal{R}\left[Q_1^+ - H^T Q_2Q_1^+, H^T\right]e.
\end{equation}
\begin{remark}The concept behind the RR method is to transform the operations of reduced biquaternion matrices into the corresponding operations of the first row block of the real representation matrices. This approach takes full advantage of the special structure of real representation matrices, consequently minimizing the count of floating-point operations required. Moreover, our method exclusively employs real matrices and real arithmetic operations, avoiding the complexities associated with complex matrices and complex arithmetic operations used in the CR method. Consequently, our approach is more efficient and time-saving compared to the CR method.
\end{remark}
\section{Application}\label{sec4}
We will now employ the framework developed in Section \ref{sec3} for finding the least squares Hermitian solution with the least norm of CME \eqref{eq1.1}. The problem can be formulated as follows:
\begin{problem} \label{pb2}
	Given matrices $A,C \in \C^{m \times n}$, $B,D \in \C^{n \times s}$, and $E,F \in \C^{m \times s}$, let
	\begin{equation*}
		\widetilde{\mathcal{H}}_{LE}= \left\{ X \; | \; X \in \HC^{n \times n}, \;  \norm{\left(AXB-E, CXD-F\right)}_F= \min_{\widetilde{X} \in \HC^{n \times n}} \norm{\left(A\widetilde{X}B-E, C\widetilde{X}D-F\right)}_F \right\}.
	\end{equation*}
	Then, find $\widetilde{X}_H \in \widetilde{\mathcal{H}}_{LE}$ such that
	$$\norm{\widetilde{X}_H}_F= \underset{X \in \widetilde{\mathcal{H}}_{LE}}{\min}\norm{X}_F.$$
\end{problem}

It is important to emphasize that complex matrix equations are particular cases of reduced biquaternion matrix equations. Let $A=A_0+A_1\i+A_2\j+A_3\k \in \QR^{m \times n}$, then $A \in \C^{m \times n} \iff A_2=0$, $A_3=0$. Consequently, we can apply the framework established in Section \ref{sec3} to address Problem \ref{pb2}. We take the real representation of matrix $A=A_0+A_1\i \in \C^{m \times n}$ denoted by $\widetilde{A}^R$ as 
\begin{equation}\label{eq4.1}
	\widetilde{A}^R = \begin{bmatrix}
		A_0 & -A_1 \\
		A_1 & A_0
	\end{bmatrix},
\end{equation}
and the first block row of the block matrix $\widetilde{A}^R$ as $\widetilde{A}^R_r=\left[A_0, -A_1\right]$. 

For matrix $X=X_0+X_1\i \in \C^{n \times n}$, we have
\begin{equation}\label{eq4.2}
	\vec(\widetilde{X}_r^R)=\begin{bmatrix}
		\vec(X_0) \\
		-\vec(X_1)
	\end{bmatrix}.
\end{equation}
At first, we need a result establishing the relationship between $\vec(\widetilde{X}^R)$ and $\vec(\widetilde{X}_r^R)$. Following a similar approach to Lemma \ref{lem2.2}, we get $\vec(\widetilde{X}^R)=\widetilde{\mathcal{J}}\vec(\widetilde{X}_r^R)$, where \begin{equation}\label{eq4.3}
	\widetilde{\mathcal{J}}=\begin{bmatrix}
	\widetilde{\mathcal{J}}_0   & 	-\widetilde{\mathcal{J}}_1 \\
	\widetilde{\mathcal{J}}_1   &     	\widetilde{\mathcal{J}}_0
\end{bmatrix}.
\end{equation} We have $\widetilde{\mathcal{J}}_0=\begin{bmatrix}
\widetilde{\mathcal{J}}_{01} \\
\widetilde{\mathcal{J}}_{02} \\
\vdots \\
\widetilde{\mathcal{J}}_{0n} \\
\end{bmatrix}$ and
$\widetilde{\mathcal{J}}_1=\begin{bmatrix}
\widetilde{\mathcal{J}}_{11} \\
\widetilde{\mathcal{J}}_{12} \\
\vdots \\
\widetilde{\mathcal{J}}_{1n} \\
\end{bmatrix}$.
 $\widetilde{\mathcal{J}}_{ij}$ is a block matrix of size $2 \times n$ with $I_n$ at $(i+1,j)^{th}$ position, and the rest of the entries are zero matrices of size $n$, where $i=0,1$ and $j=1,2,\ldots,n$.
 
 To enhance our understanding of above mentioned result, we will examine it in the context of $n = 2$. In this scenario, we have
 $$\widetilde{\mathcal{J}}_0=\begin{bmatrix}
 	\widetilde{\mathcal{J}}_{01} \\
 		\widetilde{\mathcal{J}}_{02}
 \end{bmatrix}=\begin{bmatrix}
 	I_2       &   0        \\
 	0          &   0         \\
 	\hdashline
 	0          &  I_2         \\
 	0          &   0            \\
       \end{bmatrix}, \;
 \widetilde{\mathcal{J}}_1=\begin{bmatrix}
 		\widetilde{\mathcal{J}}_{11} \\
 		\widetilde{\mathcal{J}}_{12}
 \end{bmatrix}=\begin{bmatrix}
 	0       &   0           \\
 	I_2          &   0      \\
 \hdashline
 	0          &  0            \\
 	0          &   I_2         
 	\end{bmatrix}.$$
Next, we will determine the expression for $\vec(\widetilde{X}_r^R)$ when $X \in \HC^{n \times n}$. Following a similar approach to Lemma \ref{lem2.6}, we have $X=X_0+X_1\i \in \HC^{n \times n} \iff \vec(\widetilde{X}_r^R) = \widetilde{\mathcal{Q}}\begin{bmatrix}
 	\vec_S(X_0) \\
 	\vec_A(X_1)
 \end{bmatrix} $, where 
\begin{equation}\label{eq4.4}
	\widetilde{\mathcal{Q}}=\begin{bmatrix}
    K_S  &  0  \\
    0       & -K_A 
\end{bmatrix}.
\end{equation}
Before proceeding, we introduce some notations that will be used in the subsequent results.
\begin{equation}\label{eq4.5}
	\widetilde{\mathcal{P}} = 	\begin{bmatrix}
		(\widetilde{B}^R)^T \otimes \widetilde{A}_r^R \\
		(\widetilde{D}^R)^T \otimes \widetilde{C}_r^R
	\end{bmatrix}, \; \;  \widetilde{\mathcal{R}}=\diag(K_S, K_A).
\end{equation}
\begin{theorem}\label{thm4.1}
	Given matrices $A,C \in \C^{m \times n}$, $B,D \in \C^{n \times s}$, and $E,F \in \C^{m \times s}$, let $X=X_0+X_1\i  \in \C^{n \times n}$. Let $\widetilde{\mathcal{P}}$ and $\widetilde{\mathcal{R}}$ be of the form \eqref{eq4.5}, and let $\widetilde{\mathcal{J}}$ and $\widetilde{\mathbb{Q}}$ be of the form \eqref{eq4.3} and \eqref{eq4.4}, respectively. Then, the set $\widetilde{\mathcal{H}}_{LE}$ of Problem \ref{pb2} can be expressed as
	\begin{equation}\label{eq4.6}
		\widetilde{\mathcal{H}}_{LE}=\left\{X \; \left| \; \begin{bmatrix}
			\vec(X_0) \\
			\vec(X_1)  
		\end{bmatrix}\right.=\widetilde{\mathcal{R}}\left(\widetilde{\mathcal{P}}\widetilde{\mathcal{J}}\widetilde{\mathcal{Q}}\right)^+
		\begin{bmatrix}
			\vec(\widetilde{E}_r^R) \\
			\vec(\widetilde{F}_r^R)
		\end{bmatrix} + \widetilde{\mathcal{R}}\left(I_{n^2}-\left(\widetilde{\mathcal{P}}\widetilde{\mathcal{J}}\widetilde{\mathcal{Q}}\right)^+\left(\widetilde{\mathcal{P}}\widetilde{\mathcal{J}}\widetilde{\mathcal{Q}}\right)\right)y
		\right\},
	\end{equation}
	where $y \in \R^{n^2}$ is an arbitrary vector. Furthermore, the unique solution $\widetilde{X}_{H} \in \widetilde{\mathcal{H}}_{LE}$ to Problem \ref{pb2} satisfies
	\begin{equation}\label{eq4.7}
		\begin{bmatrix}
			\vec(X_0) \\
			\vec(X_1)  
		\end{bmatrix}= \widetilde{\mathcal{R}}\left(\widetilde{\mathcal{P}}\widetilde{\mathcal{J}}\widetilde{\mathcal{Q}}\right)^+	\begin{bmatrix}
			\vec(\widetilde{E}_r^R) \\
			\vec(\widetilde{F}_r^R)
		\end{bmatrix}. 
	\end{equation}	
\end{theorem}
\begin{proof}
	The proof follows along similar lines as Theorem \ref{thm1}. 
\end{proof}
\begin{theorem}\label{thm4.2}
	Let $\widetilde{\mathcal{P}}$ and $\widetilde{\mathcal{R}}$ be in the form of \eqref{eq4.5}, and let $X=X_0+X_1\i \in \C^{n \times n}$. Additionally, let $\widetilde{\mathcal{J}}$ and $\widetilde{\mathbb{Q}}$ be as in \eqref{eq4.3} and \eqref{eq4.4}, respectively. Then, the CME \eqref{eq1.1} has a Hermitian solution $X\in \HC^{n \times n}$ if and only if
	\begin{equation}\label{eq4.8}
		\left(I_{4ms}-\left(\widetilde{\mathcal{P}}\widetilde{\mathcal{J}}\widetilde{\mathcal{Q}}\right)\left(\widetilde{\mathcal{P}}\widetilde{\mathcal{J}}\widetilde{\mathcal{Q}}\right)^+\right)\begin{bmatrix}
			\vec(\widetilde{E}_r^R) \\
			\vec(\widetilde{F}_r^R)
		\end{bmatrix}=0.
	\end{equation}	
	In this case, the general solution $X \in  \HC^{n \times n}$ satisfies
	\begin{equation}\label{eq4.9}
		\begin{bmatrix}
			\vec(X_0) \\
			\vec(X_1) 
		\end{bmatrix} = \widetilde{\mathcal{R}}\left(\widetilde{\mathcal{P}}\widetilde{\mathcal{J}}\widetilde{\mathcal{Q}}\right)^+\begin{bmatrix}
			\vec(\widetilde{E}_r^R) \\
			\vec(\widetilde{F}_r^R)
		\end{bmatrix}+\widetilde{\mathcal{R}}\left(I_{n^2}-\left(\widetilde{\mathcal{P}}\widetilde{\mathcal{J}}\widetilde{\mathcal{Q}}\right)^+\left(\widetilde{\mathcal{P}}\widetilde{\mathcal{J}}\widetilde{\mathcal{Q}}\right)\right)y,
	\end{equation}
	where $y \in \R^{n^2}$ is an arbitrary vector. Furthermore, if \eqref{eq4.8} holds, then the CME \eqref{eq1.1} has a unique solution $X \in \HC^{n \times n}$ if and only if 
	\begin{equation}\label{eq4.10}
		\rank \left(\widetilde{\mathcal{P}}\widetilde{\mathcal{J}}\widetilde{\mathcal{Q}}\right)=n^2.
	\end{equation}
	In this case, the unique solution $X \in  \HC^{n \times n}$ satisfies
	\begin{equation}\label{eq4.11}
		\begin{bmatrix}
			\vec(X_0) \\
			\vec(X_1) 
		\end{bmatrix} = \widetilde{\mathcal{R}}\left(\widetilde{\mathcal{P}}\widetilde{\mathcal{J}}\widetilde{\mathcal{Q}}\right)^+\begin{bmatrix}
			\vec(\widetilde{E}_r^R) \\
			\vec(\widetilde{F}_r^R)
		\end{bmatrix}. 
	\end{equation}
\end{theorem}
\begin{proof}
	The proof follows along similar lines as Theorem \ref{thm2}.
\end{proof}
	Our next step is to illustrate how our developed method can solve inverse problems. Here, we examine inverse problems where the spectral constraint involves only a few eigenpair information rather than the entire spectrum. Mathematically, the problem statement is:
\begin{problem}[PDIEP for Hermitian matrix]\label{pb3}
	Given vectors $\{u_1,u_2,\ldots,u_k\} \subset \mathbb{C}^n$, values $\{\lambda_1, \lambda_2, \ldots, \lambda_k\} \subset \mathbb{R}$, find a Hermitian matrix $M \in \HC^{n \times n}$ such that
	\begin{equation*}
		Mu_i = \lambda_i u_i, \; \; \; \; i=1,2, \ldots, k.
	\end{equation*}
\end{problem}
To simplify the discussion, we will use the matrix pair $\left(\Lambda, \Phi\right)$ to describe partial eigenpair information, where
\begin{equation}\label{eq4.12}
	\Lambda= \diag(\lambda_1,\lambda_2,\ldots,\lambda_k) \in \mathbb{R}^{k \times k} , \; \mbox{and} \; \Phi= [u_1,u_2,\ldots,u_k] \in \mathbb{C}^{n \times k} .
\end{equation}
Problem \ref{pb3} can be written as $M\Phi=  \Phi \Lambda$. 
Before moving forward, we will introduce certain notations that will be employed in the subsequent result.
\begin{equation}\label{eq4.13}
\widetilde{\mathcal{S}}= 
		(\widetilde{B}^R)^T \otimes \widetilde{A}_r^R , \; \;  \widetilde{R}=\diag(K_S, K_A), \; \; N=\widetilde{\mathcal{S}}\widetilde{\mathcal{J}}\widetilde{\mathcal{Q}}, \; \; t= 	\vec(\widetilde{E}_r^R) .
\end{equation}
\begin{corollary}\label{cor4.1}
	Given a pair $\left(\Lambda, \Phi\right) \in \mathbb{R}^{k \times k} \times \mathbb{C}^{n \times k} \left(k \leq n\right)$ in the form of \eqref{eq4.12}. Let $\widetilde{\mathcal{S}}$, $\widetilde{\mathcal{R}}$, $N$, and $t$ be in the form of \eqref{eq4.13}, and let $M=M_0+M_1\i \in \C^{n \times n}$. Additionally, let $\widetilde{\mathcal{J}}$ and $\widetilde{\mathbb{Q}}$ be as in \eqref{eq4.3} and \eqref{eq4.4}, respectively. If $\rank(N)= \rank([N, t])$, then the general solution to Problem \ref{pb3} satisfies
	\begin{equation}\label{eq4.14}
	\begin{bmatrix}
		\vec(M_0) \\
		\vec(M_1) 
	\end{bmatrix} = \widetilde{R}\left(\widetilde{\mathcal{S}}\widetilde{\mathcal{J}}\widetilde{\mathcal{Q}}\right)^+\left(
		\vec(\widetilde{E}_r^R) 
	\right)+\widetilde{R}\left(I_{n^2}-\left(\widetilde{\mathcal{S}}\widetilde{\mathcal{J}}\widetilde{\mathcal{Q}}\right)^+\left(\widetilde{\mathcal{S}}\widetilde{\mathcal{J}}\widetilde{\mathcal{Q}}\right)\right)y,
\end{equation}
where $y \in \R^{n^2}$ is an arbitrary vector. 
\end{corollary}
\begin{proof}
	By using the transformations
	\begin{equation*}
		A= I_{n}, \; X= M, \; B= \Phi, \; \mbox{and} \; E= \Phi \Lambda, 
	\end{equation*}
	we can find solution to Problem \ref{pb3} by solving matrix equation $AXB= E$ for $X \in \HC^{n \times n}$, which is a special case of matrix equation \eqref{eq1.1} over the complex field. 
\end{proof}
\begin{remark}
	We can choose $y$ arbitrarily in Corollary \ref{cor4.1}; therefore $M$ need not be unique.  
\end{remark}
\section{Numerical Verification}\label{sec5}
In this section, we present numerical examples to verify our results. All calculations are performed on an Intel Core $i7-9700@3.00GHz/16GB$ computer using MATLAB $R2021b$ software. 

First, we present two numerical examples. In the first example, we compute the errors between the computed solution obtained through the RR method and the corresponding actual solution for Problem \ref{pb1}. The second example involves a comparison of the CPU time required to determine the Hermitian solution of RBME \eqref{eq1.1} using both the RR method and the CR method.  Additionally, we compare the errors between the actual solutions and the corresponding computed solutions obtained via the RR method and the CR method for the computation of the Hermitian solution of RBME \eqref{eq1.1}. In both examples, the comparison is made for various dimensions of matrices.
\begin{exam}\label{ex1}
	Let $m=n=2k$, $s=k$, for $k=1:20$. Consider the RBME $(AXB,CXD)= (E,F)$, where
	\begin{eqnarray*}
		A&=&10*rand(m,n)+rand(m,n)\i+rand(m,n)\j+rand(m,n)\k, \\ B&=&rand(n,s)+rand(n,s)\i+rand(n,s)\j+rand(n,s)\k,\\
		C&=&rand(m,n)+10*rand(m,n)\i+4*rand(m,n)\j+rand(m,n)\k,\\
		D&=&rand(n,s)+2*rand(n,s)\i+rand(n,s)\j+rand(n,s)\k.
	\end{eqnarray*}
	Let $S_0 = S_1 = rand(n,n)$, $S_2=5*rand(n,n)$, and $S_3=2*rand(n,n)$. Define
	$$\widetilde{X}=(S_0+S_0^T)+(S_1-S_1^T)\i+(S_2-S_2^T)\j+(S_3-S_3^T)\k.$$
	Let $E=A\widetilde{X}B$ and $F=C\widetilde{X}D$. Hence, $\widetilde{X}$ is the unique minimum norm least squares Hermitian solution of the RBME $(AXB,CXD)=(E,F)$.\\
	Next, we take matrices $A$, $B$, $C$, $D$, $E$, and $F$ as input, and use Theorem \ref{thm1} to calculate the unique minimum norm least squares Hermitian solution $X$ to Problem \ref{pb1}. Let the error $\epsilon= \log10\left(\norm{\widetilde{X}-X}_F\right)$. The relation between the error $\epsilon$ and $k$ is shown in Figure \ref{fig1}.
\end{exam}
\begin{figure}[ht]
	\centering
	\includegraphics[width=1.00\textwidth]{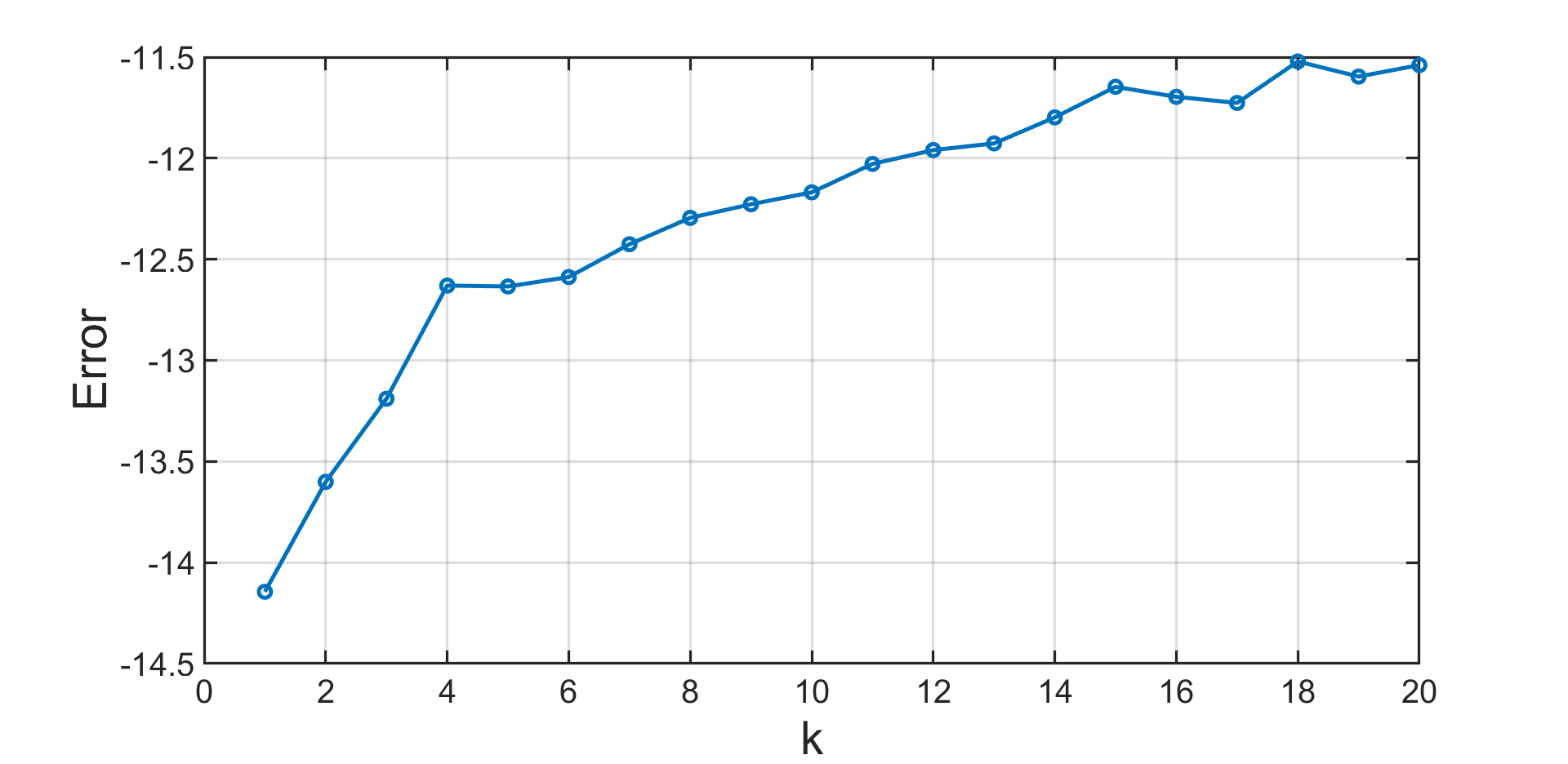} 
	\vspace{-1.0cm}
	\caption{The error in solving Problem \ref{pb1} }
	\label{fig1}
\end{figure}
Based on Figure \ref{fig1}, the error $\epsilon$ between the solution of Problem \ref{pb1} obtained using Theorem \ref{thm1} and the corresponding actual solutions, for various dimensions of matrices, consistently remains less than or equal to -11.5. This demonstrates the effectiveness of RR method in computing solution for Problem \ref{pb1}.
\begin{exam}
	Let $n=2k$, $m=n+16$, $s=n+6$, for $k=1:20$. Consider the RBME $(AXB,CXD)= (E,F)$, where
	\begin{eqnarray*}
		A&=& zeros(m,n)+\begin{bmatrix}
			eye(n) \\
			zeros(16,n)
		\end{bmatrix}\i+zeros(m,n)\j+zeros(m,n)\k, \\
		B &=& zeros(n,s)+zeros(n,s)\i+zeros(n,s)\j+\begin{bmatrix}
			-eye(n)  & zeros(n,6)
		\end{bmatrix}\k, \\
		C &=& zeros(m,n)+zeros(m,n)\i+\begin{bmatrix}
			eye(n) \\
			zeros(16,n)
		\end{bmatrix}\j+zeros(m,n)\k, \\
		D &=& zeros(n,s)+zeros(n,s)\i+ones(n,s)\j+zeros(n,s)\k.
	\end{eqnarray*}
	Let $S_0=zeros\left(\frac{n}{2},\frac{n}{2}\right)$, $S_1=eye\left(\frac{n}{2}\right)$, $S_2=randn(n,n)$, and $S_3=randn(n,n)$. Define
	\begin{equation*}
		\widetilde{X} = toeplitz(1:n)+\begin{bmatrix}
			S_0 &  S_1 \\
			-S_1 & S_0 
		\end{bmatrix}\i+(S_2-S_2^T)\j+(S_3-S_3^T)\k.
	\end{equation*}
	Let $E=A\widetilde{X}B$ and $F=C\widetilde{X}D$. Clearly, the RBME $(AXB,CXD)=(E,F)$ is consistent and $\widetilde{X}$ is its unique Hermitian solution. 
	
	Next, we input matrices $A$, $B$, $C$, $D$, $E$, and $F$. By employing both the RR method (Theorem \ref{thm2}) and the CR method (\cite[Theorem~3.1]{MR4137050}), we calculate the Hermitian solution for RBME \eqref{eq1.1}.  Let $X_1$ and $X_2$ denote the Hermitian solutions obtained through the RR method and the CR method, respectively. Let errors $\epsilon_1 = \log10\left(\norm{\widetilde{X}-X_1}_F\right)$ and $\epsilon_2 = \log10\left(\norm{\widetilde{X}-X_2}_F\right)$.
\end{exam}
\begin{figure}[H]
	\centering
	\includegraphics[width=1.00\textwidth]{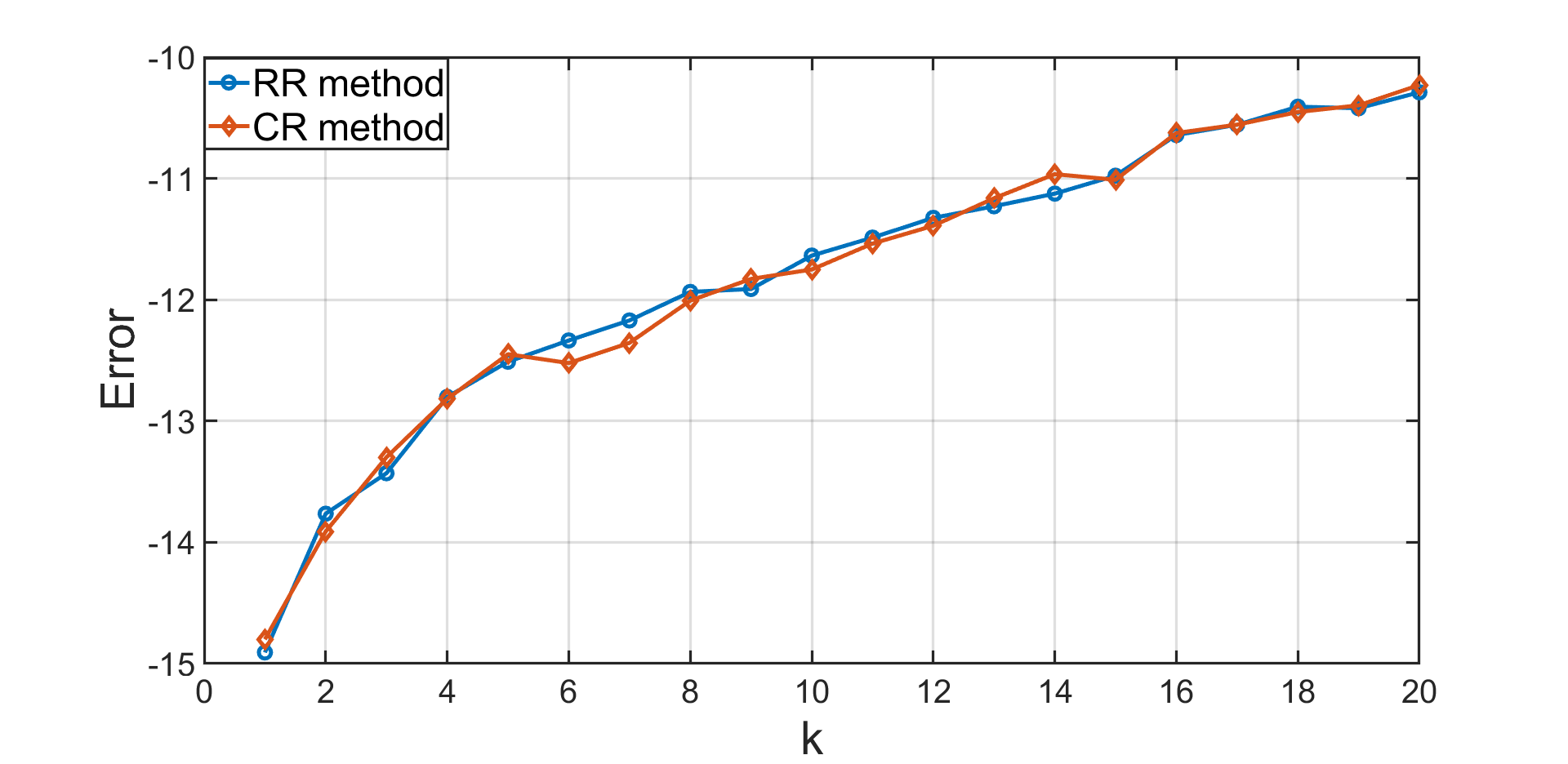} 
	\vspace{-1.0cm}
	\caption{The error in finding the Hermitian solution of RBME \eqref{eq1.1}}
	\label{fig2}
\end{figure}
Figure \ref{fig2} presents a comparison between $\epsilon_1$ and $\epsilon_2$, which is obtained by taking different value of $k$. Notably, the errors obtained from both methods are nearly identical and consistently remains less than $-10$. This demonstrates that the RR method is as effective as the CR method in determining Hermitian solutions for RBME \eqref{eq1.1}.
\begin{figure}[H]
	\centering
	\includegraphics[width=1.00\textwidth]{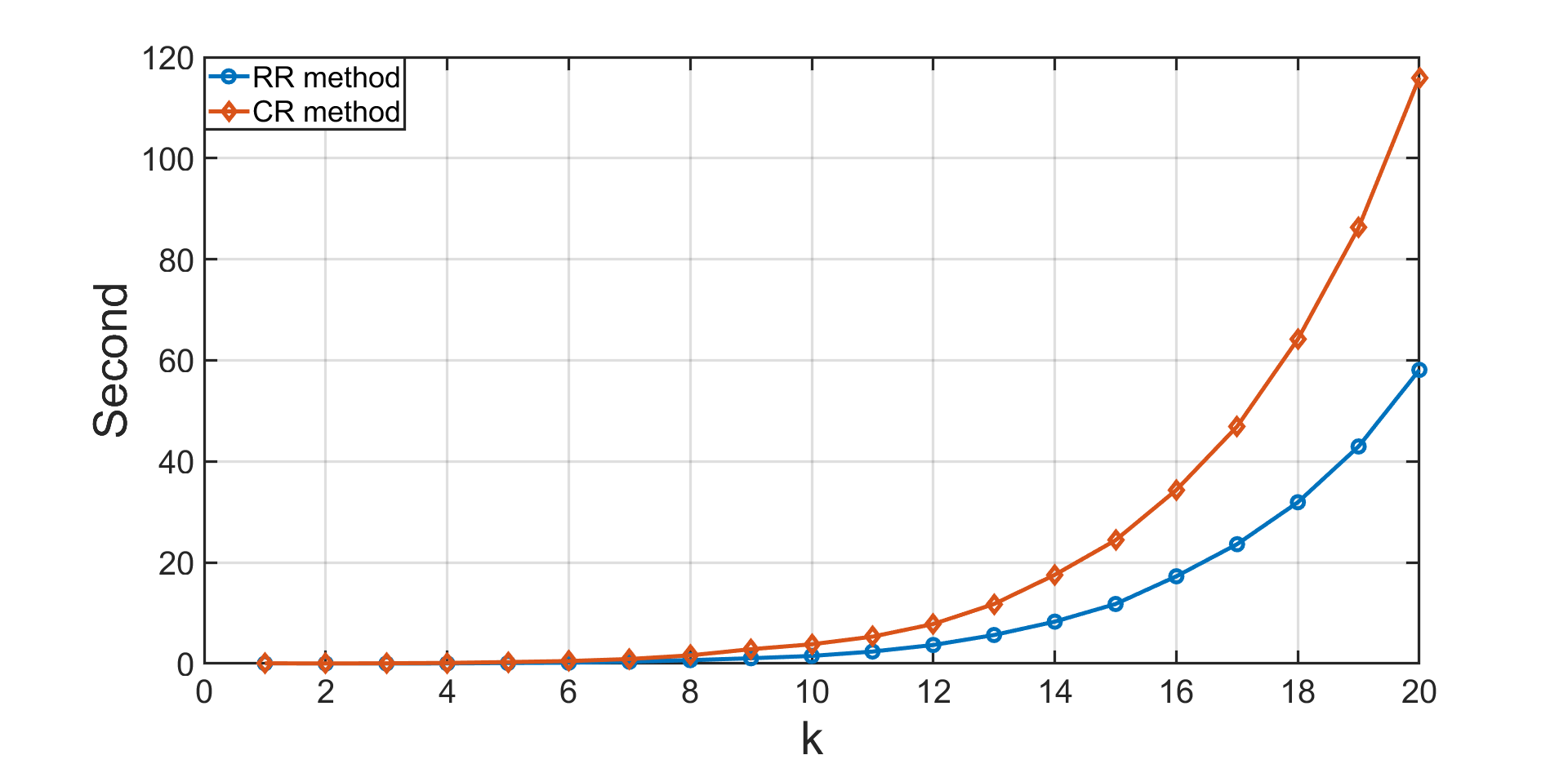} 
	\vspace{-1.0cm}
	\caption{The CPU time for finding Hermitian solution of RBME \eqref{eq1.1}}
	\label{fig3}
\end{figure}

Figure \ref{fig3} depicts a comparison of the CPU time taken by the RR method and the CR method in finding the Hermitian solution of RBME \eqref{eq1.1}. This comparison is conducted for various values of $k$. Notably, the CPU time consumed by RR method is less compared to the CR method. This highlights the enhanced efficiency and time-saving nature of our proposed RR method compared to the CR method.

Next, we illustrate two examples for solving PDIEP for a Hermitian matrix.
\begin{exam}\label{ex5.1}
	To establish test data, we first generate a Hermitian matrix $M$ and compute its eigenpairs. Let $S_0=S_1=randn(5)$, $M_0=triu(S_0)+triu(S_0,1)^T$, and $M_1=S_1-S_1^T$. Define $M=M_0+M_1\i$.
	
Reconstruction from three eigenpairs : Let the prescribed partial eigeninformation be given by $\left(\lambda_1, u_1\right)$, $\left(\lambda_2, u_2\right)$, and $\left(\lambda_3, u_3\right)$, where $\lambda_1=-5.4837$, $\lambda_2=-1.8785$, and $\lambda_3=-0.1774$. Let
	$$u_1 = \begin{bmatrix}
		0.2636 + 0.3383\i   \\
		0.4236 - 0.0100\i   \\
		-0.4570 + 0.0000\i  \\ 
		-0.1785 + 0.2168\i   \\
		-0.5906 + 0.0000\i  
	\end{bmatrix},
	u_2=\begin{bmatrix}
		0.1280 + 0.4097\i   \\
		0.4292 - 0.3207\i  \\
		0.1638 - 0.3434\i  \\
		0.5420 - 0.1538\i   \\
		0.2580 + 0.0000\i  
	\end{bmatrix}, 
	u_3= \begin{bmatrix}
		0.0128 - 0.1852\i \\
		-0.4463 + 0.0116\i\\
		-0.0575 - 0.4873\i\\
		0.4944 + 0.3518\i \\
		-0.3964 + 0.0000\i
	\end{bmatrix}.$$
	Construct the Hermitian matrix $\widehat{M}$ such that $\widehat{M}u_i = \lambda_i u_i$ for $i=1,2,3$. Let $\widehat{\Phi}=\left[u_1, u_2, u_3\right] \in \mathbb{C}^{5 \times 3}$ and $\widehat{\Lambda}= \diag(\lambda_1, \lambda_2, \lambda_3) \in \mathbb{R}^{3 \times 3}$. By using the transformations $A= I_5, \; X= \widehat{M}, \; B=\widehat{\Phi}, \; \mbox{and} \; E= \widehat{\Phi} \widehat{\Lambda}$, we find the Hermitian solution to the matrix equation $AXB=E$. We obtain 
	$$\widehat{M}= \begin{bmatrix}
		-0.7841 + 0.0000\i & -0.4732 - 1.6515\i  & 0.9901 + 0.6641\i  &-0.1203 + 0.1603\i  & 0.9418 + 0.9087\i \\
		-0.4732 + 1.6515\i  &-1.2375 + 0.0000\i   &0.7154 - 0.0670\i   & -0.0395+ 0.7054\i & 1.0881 + 0.2186\i \\
		0.9901 - 0.6641\i   &0.7154 + 0.0670\i     &-1.0019 + 0.0000\i  & -0.7084 + 0.1295\i&
		-1.9644 + 0.0076\i  \\
		-0.1203 - 0.1603\i  &-0.0395 - 0.7054\i  &-0.7084 - 0.1295\i  &-0.8146 + 0.0000\i&
		-0.8648 + 1.1681\i \\
		0.9418 - 0.9087\i   &1.0881 - 0.2186\i      &-1.9644 - 0.0076\i  &-0.8648 - 1.1681\i&
		-1.5564 + 0.0000\i 
	\end{bmatrix}. $$
	Then, $\widehat{M}$ is the desired Hermitian matrix. We have $\epsilon_1=\|\widehat{M}u_1 - \lambda_1 u_1\|_2=2.9787 \times 10^{-15}$, $\epsilon_2=\|\widehat{M}u_2 - \lambda_2 u_2\|_2= 4.0592 \times 10^{-15}$, and $\epsilon_3=\|\widehat{M}u_3 - \lambda_3 u_3\|_2= 2.5327 \times 10^{-15}$. Errors $\epsilon_1$, $\epsilon_2$, and $\epsilon_3$ are in the order of $10^{-15}$ and are negligible. Clearly, $\widehat{M}u_i = \lambda_i u_i$ for $i=1,2,3$.
\end{exam}

\begin{exam}\label{ex5.2}
	To establish test data, we first generate a Hermitian matrix $M$. Let
	\[
	M=\begin{bmatrix}
		2.30+0.00\i & 3.50+2.00\i & 1.20-1.20\i & 2.30-3.00\i & 4.50+2.00\i \\
		3.50-2.00\i & 8.90+0.00\i & 2.35-4.00\i & 4.35+5.20\i & 6.29-6.80\i \\
		1.20+1.20\i & 2.35+4.00\i & 2.00+0.00\i & 4.30-4.20\i & 6.20-7.30\i \\
		2.30+3.00\i & 4.35-5.20\i & 4.30+4.20\i & 1.00+0.00\i & 3.00-8.60\i \\
		4.50-2.00\i & 6.29+6.80\i & 6.20+7.30\i & 3.00+8.60\i & 6.40+0.00\i 
		\end{bmatrix}.
	\]
	Let $\left(\Lambda, \Phi\right)$ denote its eigenpairs. We have $\Lambda= \diag(\lambda_1,\ldots, \lambda_5) \in \mathbb{R}^{5 \times 5}$ and $\Phi=\left[u_1, u_2, u_3, u_4, u_5\right] \in \mathbb{C}^{5 \times 5}$, where 
	$$\left[\lambda_1, \lambda_2, \lambda_3, \lambda_4, \lambda_5\right]=\left[-12.0692, \, -6.3846 , \, 3.8845 , \, 8.8081, \, 26.3613\right],$$ and their corresponding eigenvectors\\ 
	
	$\; \; \; \; \; \; \; \; \; u_1= \begin{bmatrix}
		-0.3035-0.1808\i \\
		0.0558-0.0876\i \\
		-0.4616+0.1624\i  \\
		0.1136+0.5876\i \\	
		0.5165+0.0000\i  
	\end{bmatrix}, u_2 =  \begin{bmatrix}
	-0.0754+0.2769\i \\
	0.3253-0.4400\i  \\
	-0.1970-0.4505\i \\
	0.3006+0.2681\i \\
	-0.4628+0.0000\i
\end{bmatrix}, u_3 = \begin{bmatrix}
  -0.1272-0.8033\i \\
   0.2812-0.1182\i  \\
    -0.0469+0.2233\i \\
    0.2046-0.2034\i \\
     -0.3321+0.0000\i 
\end{bmatrix},$\\

$ \; \; \; \; \; \; \; \; \;  \; \; \; \; \;  \; \; \; \; \;  \; \; \; \; \;  \; \; \; \; \;  \; \; \; \; \; u_4 = \begin{bmatrix}
  -0.2920- 0.1086\i \\
  -0.1898+ 0.5073\i \\
   -0.1913- 0.5563\i \\
    0.4244- 0.2667\i \\
    0.1110+ 0.0000\i
\end{bmatrix}, u_5 =\begin{bmatrix}
		0.1779-0.0527\i \\
		 0.3751-0.4033\i \\
	    0.2429-0.2485\i \\
	    0.1608-0.3452\i \\
		0.6296+0.0000\i
	\end{bmatrix}.
	$\\
	
		\textbf{Case $\mathbf{1}$.} Reconstruction from one eigenpair $(k=1)$:
	Let the prescribed partial eigeninformation be given by 
	\[
	\widehat{\Lambda} =\lambda_4 \in \mathbb{R} \; \textrm{and} \; \; \widehat{\Phi} = u_4 \in \mathbb{C}^{5 \times 1}.
	\]
	Construct the Hermitian matrix $\widehat{M}$ such that $\widehat{M}u_i = \lambda_i u_i$ for $i=4$. By using the transformations $A= I_5, \; X= \widehat{M}, \; B=\widehat{\Phi}, \; \mbox{and} \; E= \widehat{\Phi} \widehat{\Lambda}$, we find the Hermitian solution to the matrix equation $AXB=E$. We obtain
	\[\widehat{M}=
	\begin{bmatrix}
		0.3945+0.0000\i & 0.0032+1.5622\i & 1.1249-1.3710\i & -0.8515-1.1113\i & 	-0.2523-0.0938\i \\
		0.0032-1.5622\i & 1.5238+0.0000\i & -2.6575-2.1895\i & -2.1790+ 1.6625\i & -0.1878+0.5019\i \\
		1.1249+1.3710\i & -2.6575+2.1895\i & 1.9423+0.0000\i & 0.7065- 3.0186\i & 	-0.1981-0.5762\i \\
		-0.8515+1.1113\i & -2.1790-1.6625\i & 0.7065+3.0186\i & 1.2314+ 0.0000\i & 	0.4061-0.2552\i \\
		-0.2523+0.0938\i & -0.1878-0.5019\i & -0.1981+0.5762\i & 0.4061+ 0.2552\i & 0.0458+0.0000\i
		\end{bmatrix}.
	\]
		Then, $\widehat{M}$ is the desired Hermitian matrix. \\
		
			\textbf{Case $\mathbf{2}$.} Reconstruction from two eigenpairs $(k=2)$:
		Let the prescribed partial eigeninformation be given by 
		\[
		\widehat{\Lambda} =\diag(\lambda_2, \lambda_5) \in \mathbb{R}^{2 \times 2} \; \textrm{and} \; \; \widehat{\Phi} = \left[u_2, u_5\right]\in \mathbb{C}^{5 \times 2}.
		\]
		Construct the Hermitian matrix $\widehat{M}$ such that $\widehat{M}u_i = \lambda_i u_i$ for $i=2,5$. By using the transformations $A= I_5, \; X= \widehat{M}, \; B=\widehat{\Phi}, \; \mbox{and} \; E= \widehat{\Phi} \widehat{\Lambda}$, we find the Hermitian solution to the matrix equation $AXB=E$. We obtain
		\[\widehat{M}=
		\begin{bmatrix}
			-0.0573+0.0000\i &  2.9931 + 1.2268\i &   2.2043 + 1.6423\i &   0.5577 + 0.2617\i & 	3.1106 - 0.5499\i \\
			2.9931 - 1.2268\i &   5.1050 + 0.0000\i &   4.7608 - 1.0023\i &   5.7820 + 2.4404\i & 7.6635 - 8.2068\i \\
			2.2043 - 1.6423\i &   4.7608 + 1.0023\i &   0.2463 + 0.0000\i &   4.4811 + 1.4877\i & 	4.1500 - 5.1284\i \\
			0.5577 - 0.2617\i &   5.7820 - 2.4404\i &   4.4811 - 1.4877\i &   1.0747 + 0.0000\i & 3.3003 - 6.2445\i \\
			3.1106 + 0.5499\i &   7.6635 + 8.2068\i &   4.1500 + 5.1284\i &   3.3003 + 6.2445\i & 	7.7224 + 0.0000\i 
	\end{bmatrix}.
		\]
		Then, $\widehat{M}$ is the desired Hermitian matrix. \\
		
			\textbf{Case $\mathbf{3}$.} Reconstruction from three eigenpairs $(k=3)$:
		Let the prescribed partial eigeninformation be given by 
		\[
		\widehat{\Lambda} =\diag(\lambda_1, \lambda_3, \lambda_5) \in \mathbb{R}^{3 \times 3} \; \textrm{and} \; \; \widehat{\Phi} = \left[u_1, u_3, u_5\right]\in \mathbb{C}^{5 \times 3}.
		\]
		Construct the Hermitian matrix $\widehat{M}$ such that $\widehat{M}u_i = \lambda_i u_i$ for $i=1,3,5$. By using the transformations $A= I_5, \; X= \widehat{M}, \; B=\widehat{\Phi}, \; \mbox{and} \; E= \widehat{\Phi} \widehat{\Lambda}$, we find the Hermitian solution to the matrix equation $AXB=E$. We obtain
		\[\widehat{M}=
		\begin{bmatrix}
			 1.3834 + 0.0000\i &   3.3377 - 0.0836\i &  -0.4340 - 0.2088\i   & 3.0774 - 1.0913\i &  	5.1844 + 1.9720\i \\
			3.3377 + 0.0836\i &   5.5549 + 0.0000\i &   5.6800 - 1.2204\i &  6.9930 + 2.8290\i & 6.3976 - 7.2483i \\
			-0.4340 + 0.2088\i &   5.6800 + 1.2204\i &   0.1969 + 0.0000\i &   2.5946 - 2.2959\i & 6.6286 - 5.5009\i \\
			3.0774 + 1.0913\i &   6.9930 - 2.8290\i &   2.5946 + 2.2959\i &  -0.6425 + 0.0000\i & 	1.6758 - 8.5541\i \\
			5.1844 - 1.9720\i &   6.3976 + 7.2483\i &   6.6286 + 5.5009\i &   1.6758 + 8.5541\i & 	6.7611 + 0.0000\i
				\end{bmatrix}.
		\]
		Then, $\widehat{M}$ is the desired Hermitian matrix. 
			\begin{table}[H]
			\centering
			\begin{tabular}{cccccc}
				\toprule%
				\multicolumn{2}{c}{\textbf{Case $\mathbf{1}$ $(k=1)$}} & \multicolumn{2}{c}{\textbf{Case $\mathbf{2}$ $(k=2)$}} &
				\multicolumn{2}{c}{\textbf{Case $\mathbf{3}$ $(k=3)$}}\\
				\cmidrule(lr){1-2}  \cmidrule(lr){3-4} \cmidrule(lr){5-6}%
				Eigenpair & $\norm{\widehat{M} u_i- \lambda_i u_i}_2$ & Eigenpairs &  $\norm{\widehat{M} u_i- \lambda_i u_i}_2$ &
				Eigenpairs &  $\norm{\widehat{M} u_i- \lambda_i u_i}_2$ \\
				\midrule
				$\left(\lambda_4, u_4\right)$& $3.1005
				 \times 10^{-15}$  & $\left(\lambda_2, u_2\right)$   & $1.7590 \times 10^{-14}$ &
				$\left(\lambda_1, u_1\right)$   & $2.5702 \times 10^{-14}$  \\
				&  & $\left(\lambda_5, u_5\right)$  & $1.3918 \times 10^{-14}$ &
				$\left(\lambda_3, u_3\right)$   & $2.3303 \times 10^{-15}$ \\ 
					&  &   &  &
				$\left(\lambda_5, u_5\right)$   & $3.1112 \times 10^{-15}$ \\ 
				\bottomrule
			\end{tabular}
			\caption{Residual $\norm{\widehat{M} u_i- \lambda_i u_i}_2$ for Example \ref{ex5.2}} \label{tab1}
		\end{table}
\end{exam}
From Table \ref{tab1}, we find that the residual $\norm{\widehat{M} u_i- \lambda_i u_i}_2$, for $i=4$ in Case $1$, for $i=2,5$ in Case $2$, and for $i=1,3,5$ in Case $3$, is in the order of $10^{-14}$ and is negligible. This demonstrates the effectiveness of our method in solving Problem \ref{pb3}.

\section{Conclusion}
	We have introduced an efficient method to obtain the least squares Hermitian solutions of the reduced biquaternion matrix equation $(AXB, CXD) = (E, F)$. Our method involved transforming the constrained reduced biquaternion least squares problem into an equivalent unconstrained least squares problem of a real matrix system. This was achieved by utilizing the real representation matrix of the reduced biquaternion matrix and leveraging its specific structure. Additionally, we have determined the necessary and sufficient conditions for the existence and uniqueness of the Hermitian solution, along with a general form of the solution. These conditions and the general form of the solution were previously derived by Yuan et al. $(2020)$ using the complex representation of reduced biquaternion matrices. In comparison, our approach exclusively involved real matrices and utilized real arithmetic operations, resulting in enhanced efficiency. We have also used our developed method to solve partially described inverse eigenvalue problems over the complex field. Finally, we have provided numerical examples to demonstrate the effectiveness of our method and its superiority over the existing method. 

	\bibliography{paper4g}
\end{document}